\documentclass[12pt]{amsart}
\usepackage{amsmath, amssymb, amsthm}
\usepackage{epsfig}
\usepackage{graphicx}
\usepackage{url}
\usepackage{tikz}

\usetikzlibrary{matrix,arrows}

\makeindex

\newtheorem{theorem}{Theorem}

\newtheorem{lemma}{Lemma}

\newcommand{\tr}{\mathrm{tr}}
\begin{document}

\title[Sharkovsky's Theorem -- combinatorial dynamics] {Sharkovsky's Theorem and one-dimensional combinatorial dynamics}
	
\author{Chris Bernhardt}

\maketitle

\section{Introduction}

Let $f$ be a continuous function from a space $X$ to itself. In what follows the underlying space $X$ will initially be taken to be the real line or the unit interval. In later sections the space $X$ will be taken to be circles, trees and graphs. The continuous function $f$ can be considered as {\em mapping} the space $X$ back to itself, and the function $f$ will often be called a {\em map}. The map $f$ is thought of as telling us where the points in $X$ get mapped to one unit of time later. Given a positive integer $n$, we will let $f^n$ denote the composition of $f$ with itself $n$ times. So $f^n(x)$ gives the position of $x$ after $n$ units of time. If a point $x\in X$ has the property 
 $f^n(x)=x$ we will say that $x$ is {\em periodic} with {\em period} $n$. If $n$ is the least positive integer such that  $f^n(x)=x$ we say that $x$ has {\em least period} $n$.

   Sharkovsky proved the following, now famous, theorem in 1962. It was published in \cite{S} in 1964.

\begin{theorem} Let $f: \mathbb{R} \to \mathbb{R}$ be continuous. If $f$ has a periodic point of least period $n$ then $f$ also has a periodic point of least period $m$ for any $m \triangleleft n$, where \[1 \triangleleft 2 \triangleleft 4 \triangleleft \dots \dots 28 \triangleleft 20 \triangleleft 12 \triangleleft \dots 14 \triangleleft 10 \triangleleft 6 \dots 7 \triangleleft 5 \triangleleft 3.\]

\end{theorem}

The left side of the ordering  consists of the powers of two in ascending order. For the other integers, first factor them into a product of $2$ to a power times an odd integer. If we have $2^km$ and $2^ln$, where $m$ and $n$ are odd integers that are strictly greater than one, the integer with the larger power of two is smaller with respect to the Sharkovsky ordering; if both integers have the same power of two, then the one with the larger odd factor is the smaller (i.e., to the left) with respect to the Sharkovsky ordering.
 
 This result is now very well-known and many papers in the {\em Monthly} have discussed it and it's various proofs \cite{BH}, \cite{BB}, \cite{Du1}, \cite{Du2}, \cite{E} and \cite{M}.

The proofs of this remarkable theorem have all noted that it is not just the periods of periodic orbits that must be considered, but also the order in which the points of a periodic orbit are permuted must be taken into account. This means that the proof is necessarily combinatorial in nature. This theorem is an example of a result in one-dimensional combinatorial dynamics.

In this paper we will give a proof of Sharkovsky's theorem. The proof was given in \cite{B1}, but here we will show how this proof naturally leads to looking at maps on graphs and to Sharkovsky-type theorems.
We should add that Sharkovsky also proved what is often called the converse, that given any positive integer, there exists a map that has periodic points with only the least periods that are smaller or equal in the Sharkovsky ordering than the given positive integer. Some authors refer to the theorem above as Sharkovsky's theorem, others include the converse. We will not prove the converse in this paper, but will comment on converses at the end.

One-dimensional combinatorial dynamics is a part of dynamical systems that looks at the dynamics of maps on graphs. It uses results from dynamical systems, graph theory, combinatorics and topology. There are many interesting results, but this is a fairly new area and there are undoubtedly many more interesting results that are waiting to be discovered. Many of the arguments are simple and accessible to good undergraduates. Indeed, this seems to be an excellent area for undergraduate research. The purpose of this article is to use the proof of Sharkovsky's theorem to give an introduction to the area and some of the recent results.

\section{Some tools in the proof of Sharkovsky's Theorem}

\subsection{Markov Graphs}
As stated above, there have been many proofs of Sharkovsky's Theorem, and there has been a long connection with the Monthly. The papers of Sharkovsky were not widely read in the west until Li and Yorke's renowned paper \cite{LY} {\em Period three implies chaos}, in which it was shown, among other things, that period three implied all other periods. Shortly after this Sharkovsky's theorem received widespread attention and \v{S}tefan published his proof in \cite{St}. All of the proofs at this stage were rather intricate, but a number of mathematicians felt that there should be a more straightforward method of argument. Straffin \cite{Str} gave a partial proof using directed graphs. Around 1980 several mathematicians (Block, Guckenheimer Misiurewicz and Young in \cite{BGMY} and Ho and Morris in \cite{HM}, but see \cite{M} and \cite{ALM} for more history) arrived at what is now regarded as the standard proof using directed graphs. We will look at the basic ideas. We start by looking at an example.

\begin{figure}
\begin{centering}

\begin{tikzpicture}[scale=5]

\draw[thick] (0,0) to node[below, font=\tiny]{$E_1$ }(.5 ,0);
\draw[thick] (.5,0) to node[below, font=\tiny]{$E_2$ }(1 ,0);
\draw[thick] (1,0) to node[below, font=\tiny]{$E_3$ }(1.5 ,0);

  \draw[fill=red] (0, 0) circle (.3pt);
    \draw[fill=red] (.5, 0) circle (.3pt);
      \draw[fill=red] (1, 0) circle (.3pt);
        \draw[fill=red] (1.5, 0) circle (.3pt);
        
        \draw {(0,-.05)} node[font=\tiny]{$1$};
        
        \draw {(.5,-.05)} node[font=\tiny]{$2$};
         \draw {(1,-.05)} node[font=\tiny]{$3$};
          \draw {(1.5,-.05)} node[font=\tiny]{$4$};
  
  \draw[dotted, thick, ->, blue] (0,.02) to [bend left=30] (.5,.02);
  \draw[dotted, thick, ->, blue] (.5, .02) to [bend left=30] (1,0.02);
    \draw[dotted, thick, ->, blue] (1, .02) to [bend left=30] (1.5,.02);
     \draw[dotted, thick, ->, blue] (1.5, -.02) to [bend left=30] (.02,-.02);

 \end{tikzpicture}
\caption{Map on [1,4]}
\end{centering}

\end{figure}
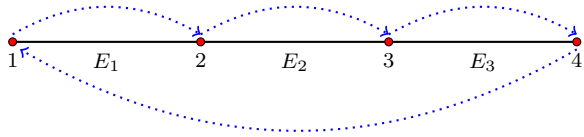 

Figure $1$ represents a map $f$ on the reals that sends $1$ to $2$, $2$ to $3$, $3$ to $4$ and $4$ back to $1$. Thus we have a periodic orbit with period $4$. The picture just shows the convex hull of the orbit, i.e., the closed interval $[1,4]$, because our arguments will be restricted to this set. The subinterval $[1,2]$ is denoted by $E_1$, $[2,3]$ is denoted $E_2$ and $[3,4]$ is denoted by $E_3$.

Associated to such a map we will draw a directed graph, called the {\em Markov Graph}. The vertices of this graph are denoted by $E_1$, $E_2$ and $E_3$. A directed edge will be drawn from vertex $E_i$ to $E_j$ if there is a closed subinterval $J$ of $E_i$ such that $f(J)=E_j$.  All the directed edges in Figure $2$ appear in the Markov Graph of any continuous map $f$ associated to the example in Figure $1$.

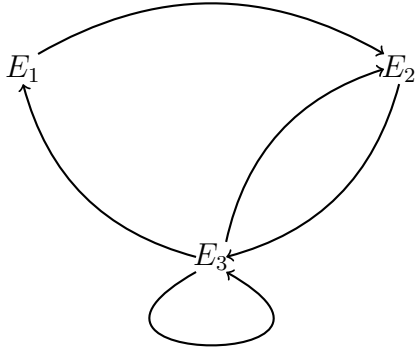
\begin{figure}
 
  \begin{centering}

\begin{tikzpicture}[scale=5]

  \draw[]{(0, 0)} node {$E_1$};
  \draw[]{(1,0)} node {$E_2$};
   \draw[]{(.5,-.5)} node {$E_3$};

  \draw[ thick, ->] (.04,.04) to [bend left=30] (.96,.04);
  \draw[ thick, ->] (1,-.04) to [bend left=30] (.54,-.5);
  \draw[ thick, ->] (.46, -.5) to [bend left=30] (0,-.04);
   \draw[ thick, ->] (.54, -.46) to [bend left=30] (.96,0);
    \draw[ thick, ->] (.46,-.54)..controls (0,-.8) and (1,-.8)..  (.54,-.54);

 \end{tikzpicture}

\caption{Markov Graph}

 \end{centering}
\end{figure}

We now look at walks in the Markov Graph. In particular, we are interested in walks that begin and end at the same vertex (closed walks) and that are non-repetitive in the sense that the walk cannot be written as the repetition of a shorter walk. So in the example, $E_3E_1E_2E_3E_2E_3$ is a non-repetitive walk from vertex $3$ to vertex $3$ of length $5$. The walk $E_3E_2E_3E_2E_3$ is a walk from $E_3$ to itself, but it is repetitive being a repetition of $E_3E_2E_3$ two times.

If we are given a walk $E_{i_1}E_{i_2}\dots E_{i_n} E_{i_1}$, then we know that there is a closed subinterval $J_n$ in $E_{i_n}$ such that $f(J_n)=E_{i_1}$. We can then find a subinterval $J_{n-1}$ of $E_{i_{n-1}}$ such that $f(J_{n-1})=J_n$ and so $f^2(J_{n-1})=E_{i_1}$. We proceed inductively until we obtain a subinterval $J_1$ of $E_{i_1}$ with the property that $f^n(J_1)=E_{i_1}$. Since $J_1 \subseteq E_{i_1}$ and $f^n(J_1)=E_{i_1}$, it follows from the Intermediate Value Theorem that $f^n$ must have a fixed point in $E_{i_1}$. There could be more than one, of course, but there must be at least one. We will call it $a$. From the construction, we also know for $1 \leq k \leq n$ that $f^k(a) \in E_k$. We would like to say that if the walk is non-repetitive then the periodic point must have least period $n$, but this is not quite true. The problem arises from the fact that $2$ belongs to both $E_1$ and $E_2$ and $3$ belongs to both $E_2$ and $E_3$. So we do need to check that our non-repetitive walk is not describing the iterates of $2$ or $3$. However, with this minor proviso, the key idea is that non-repetitive closed walks in the Markov Graph correspond to periodic points with least period equal to the length of the walk. This idea is very powerful.

In the example, notice that we can form non-repetitive walks from $E_3$ to itself of all possible length without using $E_1$. One way of doing this is $E_3E_3$, $E_3E_2E_3$, $E_3E_3E_2E_3$, $E_3E_3E_3E_2E_3$ and so on. Each of these corresponds to a periodic point. Since $E_1$ is not being used, none of the periodic orbits can correspond to iterates of $2$ which gets mapped to $1$ after three iterations. This means that for the example above, the map must have a periodic point with least period $n$ for any positive integer $n$.

 \section{Markov matrices}\index{Markov matrices} Given a directed graph, it is natural to associate a matrix to it. The {\em Markov Matrix} associated to a Markov Graph is the square matrix with entry $e_{i,j}$ equal to the number of directed edges from $E_j$ to $E_i$. 
 
 Some readers may be wondering why the matrix is defined this way and why it is not the transpose. Later we will give a brief explanation of why it makes sense to define it this way. Essentially it is due to the fact that function composition goes from right to left.
 
Powers of the matrix give information about the number of walks from one vertex to another. The following theorem tells how.

\begin{theorem}\label{powersM} Let $M$ be the Markov matrix associated to a directed graph that has vertices labeled $E_1, \dots, E_n$, then the $ij$th entry of $M^k$ gives the number of walks of length $k$ from $E_j$ to $E_i$.
\end{theorem}

Thus in the previous example its Markov Matrix is \[M=\left(\begin{array}{ccc}0 & 0 & 1 \\1 & 0 & 1 \\0 & 1 & 1\end{array}\right).\]
For example, if we are interested in walks of length $4$ we look at $M^4$ which in this case is \[\left(\begin{array}{ccc}1 & 2 & 4 \\2 & 3 & 6 \\2 & 4 & 7\end{array}\right).\] This tells us that from vertex $E_1$ in the markov graphs there are a total of $5$ walks of length $4$, one of which ends back at $E_1$, two end at $E_2$ and two at $E_3$.

In this example, we can see that there are a total of $31$ walks of length $4$.

\section{Trace and periodic orbits}

We are interested in periodic orbits and, as noted, above these correspond to the closed walks. The number of closed walks of length $k$ can be determined from the entries along the main diagonal of the powers of $M$. Thus the trace of the matrix gives important information. We illustrate by continuing with the example.

Now, \[M=\left(\begin{array}{ccc}0 & 0 & 1 \\1 & 0 & 1 \\0 & 1 & 1\end{array}\right), M^2=\left(\begin{array}{ccc}0 & 1 & 1 \\0 & 1 & 2 \\1 & 1 & 2\end{array}\right), M^4=\left(\begin{array}{ccc}1 & 2 & 4 \\2 & 3 & 6 \\2 & 4 & 7\end{array}\right).\]

The trace of $M$ is $1$, telling us that there is one closed walk of length one. i.e., the one from $E_3$ to itself. The trace of $M^2$ is $3$. The closed walks of length two are the walk from $E_3$ to $E_2$ which is counted twice - once for the closed walk considered  as starting at $E_2$ and once for the walk starting and ending at $E_3$ - and the repetition of the closed walk of length one. Thus the number of non-repetitive closed walks of length two is $\frac{1}{2}(3-1)=1$, if we don't count the same walk with different starting points as separate walks.

Similarly, the trace of $M^4$ is $11$. The number of non-repetitive closed walks of length $4$ is $\frac{1}{4}(11-3)=2$, and so $f$ must have at least two periodic orbits with least period $4$.

\section{Oriented Markov Graphs and Matrices}\index{Oriented Markov Graphs}\index{Oriented Markov Matrices} We now choose an orientation for $E_1$, $E_2$ and $E_3$. The standard orientation is from left to right, but any other orientation could have been chosen. (This point becomes relevant later when we consider edges on graphs and there is no longer a natural way of choosing orientations.) Once an orientation has been chosen we define the {\em Oriented Markov Graph} to be the Markov Graph, but with the addition of plus or minus signs on the directed edges depending on whether $f$ maps the edge in an orientation preserving or reversing way. In our example if we choose the standard orientation we obtain the graph in Figure $3$.

Given a walk in the graph, we assign it a positive orientation if it has an even number of negative edges and assign it a negative orientation otherwise.

If we are given a walk $E_{i_1}E_{i_2}\dots E_{i_n} E_{i_1}$, then we know that there is a closed subinterval $J_n$ in $E_{i_n}$ such that $f(J_n)=E_{i_1}$. If the oriented Markov Graph has a directed edge with a positive sign from $E_{i_n}$ to $E_{i_1}$, we can chose $J$ such that the initial point of $J$ gets mapped to the initial point of $E_{i_1}$ and the terminal point of $J$ gets mapped to the terminal point of $E_{i_1}$.
We can then find a subinterval $J_{n-1}$ of $E_{i_{n-1}}$ such that $f(J_{n-1})=J_n$, mapping the endpoints according to the orientation given by the directed edge from  $E_{i_{n-1}}$ to $E_{i_{n}}$, and so $f^2(J_{n-1})=E_1$. We proceed inductively until we obtain a subinterval $J_1$ of $E_{i_1}$ with the property that $f^n(J_1)=E_{i_1}$. Since $J_1 \subseteq E_{i_1}$ and $f^n(J_1)=E_{i_1}$. As before we obtain a point in $J_1$ that is fixed by $f^n$. However, notice that if the closed walk has negative orientation then the fixed point(s) must be in the interior of $J_1$, and so cannot be an endpoint of $E_{i_1}$.

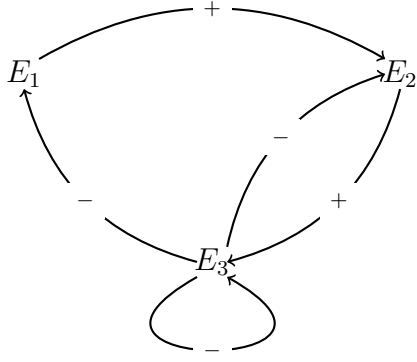
\begin{figure}
  \begin{centering}

\begin{tikzpicture}[scale=5]

  \draw[]{(0, 0)} node {$E_1$};
  \draw[]{(1,0)} node {$E_2$};
   \draw[]{(.5,-.5)} node {$E_3$};

  \draw[ thick, ->] (.04,.04) to [bend left=30] node[font=\tiny, fill=white] {$+$}(.96,.04);
  \draw[ thick, ->] (1,-.04) to [bend left=30] node[font=\tiny, fill=white] {$+$}(.54,-.5);
  \draw[ thick, ->] (.46, -.5) to [bend left=30] node[font=\tiny, fill=white] {$-$} (0,-.04);
   \draw[ thick, ->] (.54, -.46) to [bend left=30] node[font=\tiny, fill=white] {$-$}(.96,0);
    \draw[ thick, ->] (.46,-.54)..controls (0,-.8) and (1,-.8).. node[font=\tiny, fill=white] {$-$} (.54,-.54);

 \end{tikzpicture}
\caption{Oriented Markov Graph}
 \end{centering}
 \end{figure}

The {\em Oriented Markov Matrix} is the matrix with entry $e_{i,j}$ equal to the number of positive directed edges from $E_j$ to $E_i$ minus the number of negative directed edges from $E_j$ to $E_i$. In our example we obtain the Oriented Markov Matrix \[OM=\left(\begin{array}{ccc}0 & 0 & -1 \\1 & 0 & -1 \\0 & 1 & -1\end{array}\right).\]

Powers of this matrix give the sum of the positive walks minus the sum of the negative walks. 

\begin{theorem} \label{ompower} Let $OM$ be the oriented Markov matrix associated to a directed graph that has vertices labeled $E_1, \dots, E_n$, then the $ij$th entry of $OM^k$ gives the number of positive walks from $E_j$ to $E_i$ minus the number of negative walks from $E_j$ to $E_i$.
\end{theorem}

The Oriented Markov Matrix squared is 
\[OM^2=\left(\begin{array}{ccc}0 & -1 & 1 \\0 & -1 & 0 \\1 & -1 & 0\end{array}\right).\] 

We will use this example to show that we can get some useful information from these matrices.

The bottom right entry of $OM$ tells us that there is a negative walk from $E_3$ to itself. This means that there must be a positive walk of length $2$ from $E_3$ to itself (just repeat the walk of length one). The bottom right entry of $OM^2$ tells us that the number of positive walks of length two from $E_3$ to itself equals the number of negative walks. We can thus deduce that there must be a negative walk of length two from $E_3$ to itself. As in the previous argument, we can use these two walks to deduce the existence of periodic orbits of all positive least periods.

It might seem that it is simpler to use the Markov Matrix rather than the Oriented Markov Matrix, but it will be shown that the Oriented Markov Matrix has very nice algebraic properties.

\section{Piecewise linear maps}

Given an example of a periodic orbit such as is depicted in Figure $1$

\noindent there is a natural way to draw a piecewise linear map that has this periodic orbit. This map is sometimes called the {\em connect-the-dots} map.\index{connect-the-dot map}  It can be considered as the `simplest' map that has a periodic orbit of the given permutation. The graph of the piecewise linear map to the example above is drawn in Figure 4.

 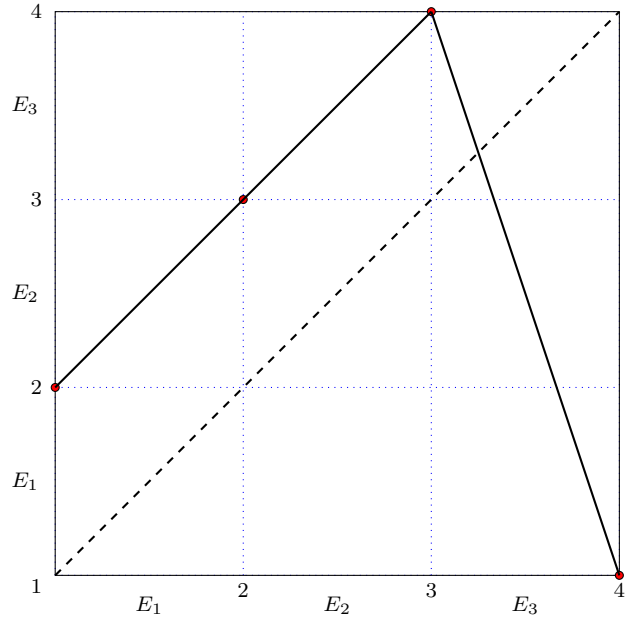
\begin{figure}
   \begin{centering}

\begin{tikzpicture}[scale=5]

\draw[step=.5, dotted, blue] (0, 0) grid (1.5,1.5);
\draw (0,0) rectangle (1.5,1.5);
  \draw[font=\tiny]{(1, -.04)} node {3};
    \draw[font=\tiny]{(1.5, -.04)} node {4};
 
  \draw[font=\tiny]{(0.5 ,-.04)} node {2};
  
  \draw[font=\tiny]{(-.05, 1)} node {3};
  \draw[font=\tiny]{(-.05, -.03)} node {1};
  \draw[font=\tiny]{(-.05, 0.5)} node {2};
   \draw[font=\tiny]{(-.05, 1.5)} node {4};
  
  \draw[fill=red] (0, .5) circle (.3pt);
    \draw[fill=red] (1, 1.5) circle (.3pt);
        \draw[fill=red] (.5, 1) circle (.3pt);
         \draw[fill=red] (1.5, 0) circle (.3pt);
   
   \draw[thick] (0, .5) to (1, 1.5);
      \draw[thick] (1, 1.5) to (1.5, 0);
      
         \draw[thick, dashed] (0, 0) to (1.5, 1.5);
         
            \draw {(.25,-.08)} node[font=\tiny]{$E_1$};
      \draw {(.75,-.08)} node[font=\tiny]{$E_2$};
         \draw {(1.25,-.08)} node[font=\tiny]{$E_3$};

            \draw {(-.08,.25)} node[font=\tiny]{$E_1$};
      \draw {(-.08,.75)} node[font=\tiny]{$E_2$};
         \draw {(-.08,1.25)} node[font=\tiny]{$E_3$};

 \end{tikzpicture}
\caption{ Graph of $L$}
 \end{centering}
 \end{figure}
 
 The graph of the composition of the map with itself is shown in Figure 5.

 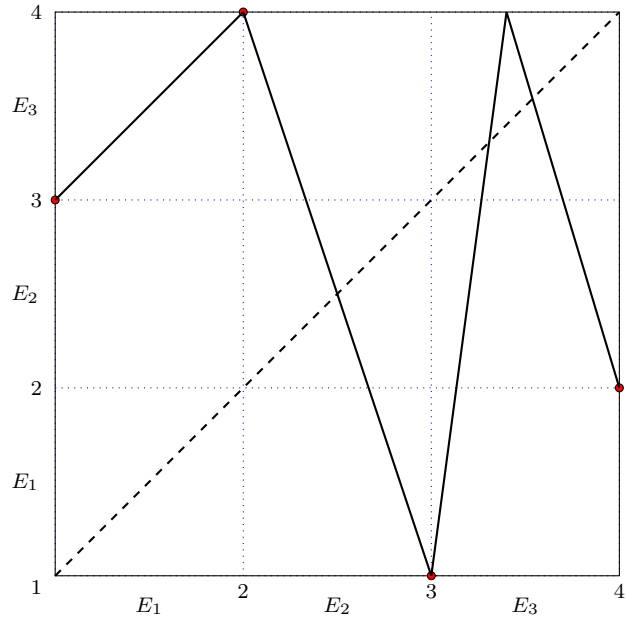
\begin{figure}
  \begin{centering}

\begin{tikzpicture}[scale=5]

\draw[step=.5, dotted, blue] (0, 0) grid (1.5,1.5);
\draw (0,0) rectangle (1.5,1.5);
  \draw[font=\tiny]{(1, -.04)} node {3};
    \draw[font=\tiny]{(1.5, -.04)} node {4};
 
  \draw[font=\tiny]{(0.5 ,-.04)} node {2};
  
  \draw[font=\tiny]{(-.05, 1)} node {3};
  \draw[font=\tiny]{(-.05, -.03)} node {1};
  \draw[font=\tiny]{(-.05, 0.5)} node {2};
   \draw[font=\tiny]{(-.05, 1.5)} node {4};
  
  \draw[fill=red] (0, 1) circle (.3pt);
    \draw[fill=red] (1, 0) circle (.3pt);
        \draw[fill=red] (.5, 1.5) circle (.3pt);
         \draw[fill=red] (1.5, .5) circle (.3pt);
   
   \draw[thick] (0, 1) to (.5, 1.5);
      \draw[thick] (.5, 1.5) to (1, 0);
       \draw[thick] (1, 0) to (1.2, 1.5);
         \draw[thick]  (1.2, 1.5) to (1.5, .5);
      
         \draw[thick, dashed] (0, 0) to (1.5, 1.5);

            \draw {(.25,-.08)} node[font=\tiny]{$E_1$};
      \draw {(.75,-.08)} node[font=\tiny]{$E_2$};
         \draw {(1.25,-.08)} node[font=\tiny]{$E_3$};

            \draw {(-.08,.25)} node[font=\tiny]{$E_1$};
      \draw {(-.08,.75)} node[font=\tiny]{$E_2$};
         \draw {(-.08,1.25)} node[font=\tiny]{$E_3$};

 \end{tikzpicture}
 \caption{Graph of $L^2$}
 \end{centering}
 \end{figure}
 
 For the moment,  we will denote the connect-the-dot map by $L$ and the composition of this map with itself by $L^2$. The map $L$ gives us information about walks of length $1$ in the Markov graph, and entries in both the Markov matrix and oriented Markov matrix. The graph of $L^2$ gives us information about walks of length $2$ in the Markov graph, and entries in both $M^2$ and $OM^2$.
 
 Looking at the graph of $L$ above $E_3$ it is clear that $E_3$ gets mapped onto $E_1E_2E_3$ with negative orientation. Thus each entry in the third column in the Markov matrix is $1$ and each entry in the Oriented Markov matrix is $-1$. 
 
 Looking at the graph of $L^2$ above $E_3$ we see that $E_3$ gets mapped to the chain $E_1E_2E_3{-E_3}{-E_2}$. If we ignore the negative signs we have one $E_1$ and two each of $E_2$ and $E_3$ which gives the third column of the square of Markov matrix as $[1, 2, 2]^T$. If we take the negative signs into account we obtain one $E_1$ and both the $E_2$ and $E_3$ terms cancel giving the third column of the square of the oriented Markov matrix as $[1, 0, 0]^T$.

 \subsection{Algebra} This section began with an example in which $1$ got mapped to $2$, $2$ to $3$, $3$ to $4$ and $4$ got mapped back to $1$. We can think of this as a permutation and describe this using cycle notation as $(1, 2, 3, 4)$. We will let $\theta=(1, 2, 3, 4)$ and will now denote the piecewise linear map associated to this by $L_\theta$ (this is the map we were calling $L$). Then $L_\theta^2$ has permutation $(1,2,3,4)^2=(1,3)(2,4)$. The following lemma is clear.
 
 \begin{lemma}
 Suppose $\theta$ is a permutation of $\{1, 2, \dots, n\}$ and that $L_\theta$ is the corresponding piecewise linear map, then $L_\theta^k$ permutes $\{1, 2, \dots, n\}$ by $\theta^k$ for any positive integer $k$.
 
 \end{lemma}

Given a $k$ we can construct $L_{\theta^k}$. Notice that both $L_{\theta^k}$ and $L_\theta^k$ have the same corresponding permutation, $\theta^k$. We illustrate with our example of $\theta=(1, 2, 3, 4)$ with $k=2$. The graph of $L_\theta^2$ is given in Figure 5 and of $L_{\theta^2}$ in Figure 6.

\begin{figure}
  \begin{centering}

\begin{tikzpicture}[scale=5]

\draw[step=.5, dotted, blue] (0, 0) grid (1.5,1.5);
\draw (0,0) rectangle (1.5,1.5);
  \draw[font=\tiny]{(1, -.04)} node {3};
    \draw[font=\tiny]{(1.5, -.04)} node {4};
 
  \draw[font=\tiny]{(0.5 ,-.04)} node {2};
  
  \draw[font=\tiny]{(-.05, 1)} node {3};
  \draw[font=\tiny]{(-.05, -.03)} node {1};
  \draw[font=\tiny]{(-.05, 0.5)} node {2};
   \draw[font=\tiny]{(-.05, 1.5)} node {4};
  
  \draw[fill=red] (0, 1) circle (.3pt);
    \draw[fill=red] (1, 0) circle (.3pt);
        \draw[fill=red] (.5, 1.5) circle (.3pt);
         \draw[fill=red] (1.5, .5) circle (.3pt);
   
   \draw[thick] (0, 1) to (.5, 1.5);
      \draw[thick] (.5, 1.5) to (1, 0);
       \draw[thick] (1, 0) to (1.5, .5);

         \draw[thick, dashed] (0, 0) to (1.5, 1.5);

            \draw {(.25,-.08)} node[font=\tiny]{$E_1$};
      \draw {(.75,-.08)} node[font=\tiny]{$E_2$};
         \draw {(1.25,-.08)} node[font=\tiny]{$E_3$};

            \draw {(-.08,.25)} node[font=\tiny]{$E_1$};
      \draw {(-.08,.75)} node[font=\tiny]{$E_2$};
         \draw {(-.08,1.25)} node[font=\tiny]{$E_3$};

 \end{tikzpicture}
 \caption{Graph of $L_{\theta^2}$}
 \end{centering}
 \end{figure}
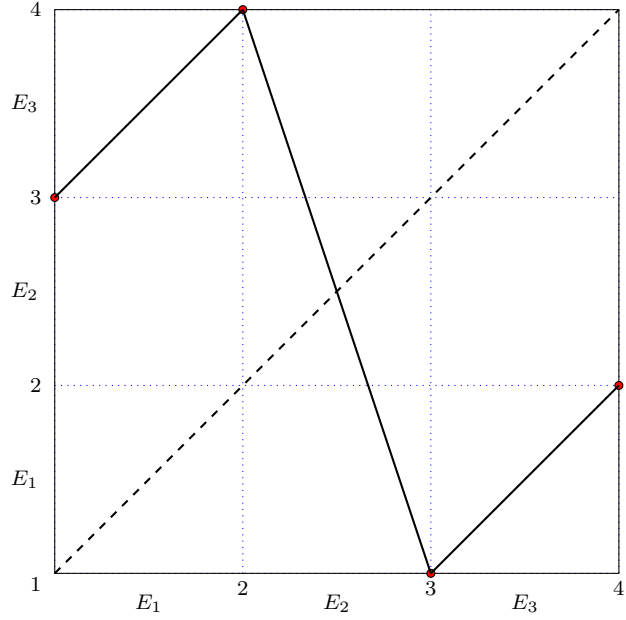

Notice that in general the Markov Matrices of $L_\theta^k$ and $L_{\theta^k}$ will not be equal. In our example, these are $\left(\begin{array}{ccc}0 & 1 & 1 \\0 & 1 & 2 \\1 & 1 & 2\end{array}\right)$ and $\left(\begin{array}{ccc}0 & 1 & 1 \\0 & 1 & 0 \\1 & 1 & 0\end{array}\right)$. However, notice that both $L_\theta^k$ and $L_{\theta^k}$ have the same Oriented Markov Matrix $\left(\begin{array}{ccc}0 & -1 & 1 \\0 & -1 & 0 \\1 & -1 & 0\end{array}\right)$. 

This insight gives the following:

\begin{theorem} \label{theta}
Suppose $\theta$ is a permutation of $\{1, 2, \dots, n\}$, then the Oriented Markov Matrices of  $L_\theta^k$ and $L_{\theta^k}$ are equal for any positive integer $k$.
\end{theorem} 

Aside: If we let $M(\theta)$ denote the Orinted Markov Graph of $L_\theta$, then the above result can be expressed as [$M(\theta)]^k=M(\theta^k)$. However, though we won't need it here, a stronger result is true. If $\alpha$ and $\beta$ are permutations of $\{1, 2, \dots, n\}$, then $M(\alpha)M(\beta)= M(\alpha \beta)$. This is the reason we defined the $(i,j)$ entry of the Oriented Markov matrix to be calculated using edges from $E_j$ to $E_i$ and not from $E_i$ to $E_j$.

\section{The trace of the oriented Markov matrix}

Let $\theta$ be a permutation of $\{1, 2, \dots, n\}$. Suppose that none of the integers $1, 2, \dots, n$ are fixed by $\theta$. Then the graph of $L_\theta$ must intersect the line $y=x$ an odd number of times. If the graph of $L_\theta$ crosses the line $y=x$ with positive orientation, the next time it crosses it will have negative orientation, and vice-versa. The first crossing must have negative orientation, and there must be an odd number of crossings. These observations give the following result.

\begin{theorem} \label{trace}\index{Trace theorem for unit interval}
Let $\theta$ be a permutation of $\{1, 2, \dots, n\}$. Suppose that none of the integers in $\{1, 2, \dots, n\}$ are fixed by $\theta$. The Oriented Markov Matrix of 
$L_\theta$ has trace of $-1$.
\end{theorem}

If $\theta^k$ fixes all of the integers in $\{1, 2, \dots, n\}$, i.e., $\theta^k$ is the identity permutation, then the graph of $L_{\theta^k}$ is just the line $y=x$ and so the Oriented Markov Matrix of 
$L_{\theta^k}$ will be the $(n-1)\times(n-1)$ identity matrix.  We state this as a theorem that will refer back to.

\begin{theorem} \label{identity}
Let $\theta$ be a permutation of $\{1, 2, \dots, n\}$. Suppose all of the integers from $1$ to $n$ are fixed by $\theta^k$, then the Oriented Markov Matrix of 
$L_{\theta^k}$ is the identity matrix.
\end{theorem}

\section{Proof of Sharkovsky's theorem -- part $1$}\index{Sharkovsky's Theorem -- proof}

Let $f: \mathbb{R} \to \mathbb{R}$ be continuous. Suppose that $f$ has a periodic point of least period $n$. By conjugating with a homeomorphism we can rescale and make the periodic orbit the integers $1$ through $n$. We will let $\theta$ denote the permutation of $\{1, 2, \dots, n\}$ given by $f$. As before we will let $L_\theta$ denote the corresponding connect-the-dots map. We have also noted that given a closed non-repetitive walk of length $m$ with negative orientation in the Oriented Markov Graph of $L_\theta$ there is a periodic point with least period $m$ for both $L_\theta$ and $f$. In this section we will prove the existence of periodic points of least period $m$ for the map $f$ by showing that the Oriented Markov Graph of $L_\theta$ has a non-repetitive closed walk of length $m$ with negative orientation.

\begin{lemma}  If $n$ is not a divisor of $2^k$ then $f$
has a periodic point of least period $2^k$.
\end{lemma} 

\begin{proof}
Since $n$ is not a divisor of $2^k$ we know that $\theta^{2^k}$ does not fix any of the integers in $\{1, 2, \dots, n\}$. Theorem \ref{trace} shows
that the trace of the Oriented Markov Matrix of $L_{\theta^{2^k}}$ is $-1$. Then Theorem \ref{theta} tells us that the trace of the Oriented Markov Matrix of $L_\theta^{2^k}$ is $-1$. This means that there is at least one negative entry on the main diagonal.
Thus by Theorem \ref{ompower} the Oriented Markov Graph has a vertex $E_j$ with a closed walk from $E_j$ to itself of length $2^k$ with negative orientation. Since the orientation is negative it cannot be the
repetition of a shorter closed walk, as any shorter closed walk would have to be repeated an even number of times. 

We have shown that there is a non-repetitive closed walk of length $2^k$ from and to $E_j$. This means that there is a closed subinterval $J \subseteq E_j$ that gets mapped onto $E_j$ by $f^{2^k}$. As pointed out before, the endpoints of $E_j$ might belong to other intervals and we have to be careful that the closed walk is not describing one of the endpoints of $E_j$. However, since $J$ gets mapped onto $E_j$ with negative orientation the point(s) that are fixed by $f^{2^k}$ must be interior point(s). Once we know that the point is not an endpoint, the fact that the walk is non-repetitive means that the fixed points under $f^{2^k}$ have minimum period of $2^k$ under $f$. \end{proof}

This lemma shows that if $m<n$, then $2^m \triangleleft 2^n$. It also shows that if $k$ is not a power of $2$ then $k$ forces all powers of $2$. So we know \[1\triangleleft 2 \triangleleft 4 \ \triangleleft 8 \triangleleft \dots\] and that all other positive integers are to the right of the dots.

\begin{lemma} \label{2kr} If $n=2^kr$, where $r>1$ is odd and $k \geq 0$. Then $f$
must have a periodic point with least period
$2^ks$ for any $s> r$.
\end{lemma}

\begin{proof}

Let $s-r=2^pq$, where $q$ is odd.
The permutation $\theta^{2^{k+p}}$ consists of $2^k$ cycles, each of which has length $r$. In particular, $\theta^{2^{k+p}}$ does not fix any of the integers, and so by 
Theorem \ref{trace} we know that the trace of the Oriented Markov Matrix of $L_{\theta^{2^{k+p}}}$ is $-1$. Since the Oriented Markov Matrix of $L_{\theta^{2^{k+p}}}$ equals the Oriented Markov Matrix of $L_\theta^{2^{k+p}}$ (Theorem \ref{theta}),  the Oriented Markov Matrix of $\theta$ raised to the $2^{k+p}$th power has trace $-1$. Thus there is a closed walk of length $2^{k+p}$ with negative orientation in the Oriented Markov Graph of $\theta$.

 Since $\theta^{2^kr}$ is the identity we know that the Oriented Markov Matrix of $\theta$ raised to the $2^kr$ power is the identity matrix, and so there is a
closed walk of length
$2^kr$ with positive orientation for each of the vertices in the oriented Markov graph of $\theta$. Thus there exists a vertex, which we will denote by $E$, that has a closed walk of
length
$2^{k+p}$ and a closed walk of length $2^kr$.

 We say a closed walk from $E$ to itself is {\em prime} if $E$ does not appear in the walk except at the end vertices.

Since the closed walk of length $2^kr$ has positive orientation, the closed walk of
length
$2^{k+p}$ has negative orientation, and $r$ is odd, it follows that these closed walks cannot just be the repetition  of the same prime closed walk. So there must be at least two prime closed walks within these two closed walks.
 We now show how a non-repetitive closed walk of length $2^ks$ can be constructed.
 
  A  closed walk with negative orientation of length $2^ks$ can be found by going  $q$ times around the closed walk of length $2^{k+p}$ and then going around the closed walk of length $2^kr$ once.
This closed walk of length $2^ks$  must contain at least two distinct prime closed walks since $q \geq 1$. Let $P$ denote a prime closed walk contained as a subpath of the  closed walk of length $2^ks$, and suppose that it appears $i$ times. Construct a new closed walk by re-arranging the prime closed walks that make the closed walk of length $2^ks$. Start with the $i$ copies of $P$ followed by the other prime closed walks in any order. This results in a non-repetitive closed walk of length $2^ks$ that belongs to the Markov graph of $L_\theta$.

 Again, we conclude that since the Oriented Markov Graph of $L_\theta$ must have a negative, non-repetitive closed walk of length $2^ks$ that $f$ must have a periodic point with least period $2^ks$.

\end{proof}

This lemma shows that  $ \dots 2^k7 \triangleleft 2^k5 \triangleleft 2^k3$ for any $k \geq 0$. So, combining the results of the two lemmas so far we obtain \[1\triangleleft 2 \triangleleft 4 \ \triangleleft 8 \triangleleft \dots \dots 2^k7 \triangleleft 2^k5 \triangleleft 2^k3. \]

To complete the proof of Sharkovsky's theorem we need to see how integers of the form $2^k$ times an odd relate to integers of the form $2^{k+1}$ times an odd. Notice that if we can show that for all $k$ that $2^k$ times an odd forces $2^{k+1}3$ then we are done, and this is exactly what we shall do. But it is worth noticing, and we will elaborate in the next section, that the previous lemma does give us more information. In the statement of the lemma there is no assumption that $s$ is odd. So for example, we can use the Lemma $3$ to show that $11$ forces $22$ and that $22$ forces $44$.

\section{First digression}
In this section we begin by re-stating what we have proved. We then comment on how this approach can be generalized to maps on graphs.

The Sharkovsky ordering can be defined as follows: 
\begin{enumerate}
\item $2^l \triangleleft 2^k$ if $k \geq l$.
\item If $n=2^ks$, where $s>1$ is odd, then

\begin{enumerate}
\item $2^l \triangleleft n$, for all non-negative integers $l$.
\item $2^kr \triangleleft n$, where $r\geq s$ and $r$ is odd.
\item $2^lr \triangleleft n$, where $l>k$ and $r>1$ is odd.

\end{enumerate}
\end{enumerate}

We can re-write Lemmas $2$ and $3$ using similar terms.

\begin{theorem}
Let $f: \mathbb{R} \to \mathbb{R}$ be continuous. If $f$ has a periodic point of least period $n$.
Then 
\begin{enumerate}
\item If $n=2^k$, then there must be periodic points of least period $2^l$ for any $l \leq k$.
\item If $n=2^ks$, where $s>1$ is odd, then

\begin{enumerate}
\item there are periodic points with minimum period $2^l$ for all positive integers $l$,
\item there are periodic points with minimum period $2^kr$ for any $r\geq s$ and $r$ is odd.
\item there are periodic points with minimum period $2^lr$ for all $l$ and $r$ satisfying: $l>k$, $r>1$ is odd, and $2^{l-k}r>s$.

\end{enumerate}
\end{enumerate}
\end{theorem}

\begin{proof}
The statements involving points with least period $2^l$ follow immediately from Lemma $2$. 

The last two statements follow from Lemma $3$.

\end{proof}

The purpose of stating the above theorem is twofold. It shows that the difference between the Sharkovsky ordering of the positive integers and the ordering given by Lemmas $2$ and $3$ is not major. The difference in the periods forced by a given $n$ with respect the two orderings only differs for positive integers less than $n$. The second reason for stating this result is that it extends to more general cases.

So far we have looked at examples similar to the one pictured in Figure 1. The underlying object, instead of being thought of as a subinterval of the real line, can be considered as a {\em graph}, in the combinatorial sense, consisting of edges and vertices. In this example, the graph would have three edges that are labeled $E_1$, $E_2$ and $E_3$ and four vertices labeled $1$, $2$, $3$ and $4$. A natural way to generalize is to consider the underlying object to be a general graph $G$, not necessarily homeomorphic to an interval,  and the map to be a map $f: G\to G$ that permutes the vertices. We will call such a map a {\em vertex map}.

For example, we could consider the following graph that is topologically a circle depicted in Figure 7.

\begin{centering}
\begin{figure}
\begin{tikzpicture}[scale=5]

\draw (0,0) circle (.4);
\draw [fill=red] (0,-.4) circle (.3pt);
\draw [fill=red] (.4,0) circle (.3pt);
\draw [fill=red] (0,.4) circle (.3pt);
\draw [fill=red] (-.4,0) circle (.3pt);
 \draw[dotted, thick, ->, blue] (-.02,.38) to [bend left=30] (-.38,0.02);
  \draw[dotted, thick, ->, blue] (-.38,-.02) to [bend left=30] (-.02,-.38);
   \draw[dotted, thick, ->, blue] (.02,-.38) to [bend left=30] (.38,-.02);
    \draw[dotted, thick, ->, blue] (.38,.02) to [bend left=30] (.02, .38);
    
        \draw {(0,.44)} node[font=\tiny]{$1$};
        
        \draw {(-.44,0)} node[font=\tiny]{$2$};
         \draw {(0,-.44)} node[font=\tiny]{$3$};
          \draw {(.44,0)} node[font=\tiny]{$4$};
          
          \draw {(-.35,.35)} node [font=\tiny]{$E_1$};
          \draw {(-.35,-.35)} node [font=\tiny]{$E_2$};
    \draw {(.35,-.35)} node [font=\tiny]{$E_3$};
    \draw {(.35,.35)} node [font=\tiny]{$E_4$};

 \end{tikzpicture}
 \caption{Circle as graph}
\end{figure}
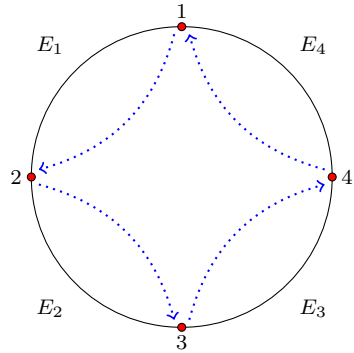
\end{centering}

\noindent

Once we move from the underlying space being the real line (or unit interval) to more general objects it becomes difficult to draw the graph of the map. However, it is clear that for the unit interval there is a natural correspondence between the connect-the-dot graphs and oriented Markov matrices. For general graphs we will usually not attempt to draw the graphs of the underlying maps, but will picture this information by using the oriented Markov matrices.

An other important change is that knowing where the vertices get mapped to does not necessarily tell us where the edges get mapped. In the example depicted in Figure 7, $E_1$ has vertices $1$ and $2$. These vertices get mapped to $2$ and $3$, respectively. But notice that there are two intervals that have these endpoints; the interval $E_2$ and the interval that goes from $2$ to $3$ in the counterclockwise direction. In order to define the underlying map we will give both the picture of the underling graph and its Oriented Markov Matrix.

 We have seen that the trace of the oriented Markov matrix gives useful dynamical information, and so we study the traces of Oriented Markov Matrices in the more general case of vertex maps on graphs. In several cases it is possible to arrive at the ordering given by the above theorem by using arguments similar to the ones we have used about the trace of the Oriented Markov Matrix.
 
 \section{Vertex maps on trees}
 
 In \cite{B2, B3} the maps on trees for which the vertices form one periodic orbit of period $n$ are studied. The arguments that we have used in this paper generalize straightforwardly to the tree case.

\begin{lemma}
Given a tree $T$ with $v$ vertices and a map $f:T \to T$ that permutes the vertices, if none of the vertices are fixed under $f$, then the trace of the Oriented Markov Matrix  is $-1$.
\end{lemma}

\begin{proof}
For each vertex, $v_i$ there is a reduced path from
$v_i$ to  $f(v_i)$. Put a dot on the first edge in this path. 

Observe that an edge $E_i$ contains two dots if and only
if $-E_i$ is in the reduced path that corresponds to $f(E_i)$. Also observe that $E_i$ contains no
dots if and only if $E_i$ is in the reduced path corresponding to $f(E_i)$. Finally, an
edge contains one dot if and only if the reduced path of $f(E_i)$ does not
contain either $E_i$ or $-E_i$. Notice that the number of dots on
the edge $E_i$ is exactly $1-M_{ii}$. If $e$ denotes the number of edges in $T$,  the total number of
dots is $\sum_1^e(1-M_{ii})=e-\tr(M)$. However, there are
exactly $v$ dots on $T$, so
$v=e-\tr(M)$, and $\tr(M)=e-v=-1$.
\end{proof}

The above result is the extension of Theorem $5$ to trees. It is straightforward to show that if all the vertices are fixed by a continuous map of a tree, then its Oriented Markov matrix is the identity. Arguments similar to those in section $8$ enable us to prover the following lemma.
 
  \begin{theorem}
Let $f: T \to T$ be a map from a tree $T$ with $n$ vertices to itself. If the vertices form one periodic orbit under $f$, 
then 
\begin{enumerate}
\item If $n=2^k$, then there must be periodic points of least period $2^l$ for any $l \leq k$.
\item If $n=2^ks$, where $s>1$ is odd, then

\begin{enumerate}
\item there are periodic points with minimum period $2^l$ for all positive integers $l$,
\item there are periodic points with minimum period $2^kr$ for any $r\geq s$ and $r$ is odd.
\item there are periodic points with minimum period $2^lr$ for all $l$ and $r$ satisfying: $l>k$, $r>1$ is odd, and $2^{l-k}r>s$.

\end{enumerate}
\end{enumerate}
\end{theorem}

\section{Vertex maps on graphs}

Similar types of arguments can also be used for vertex maps from graphs to themselves. In the case when the map is homotopic to the constant map, it is shown in \cite{BGJR}  that one obtains exactly the same ordering as in Theorem $8$, and in \cite{B5} it is shown that  the same result holds for maps on graphs that are homotopic to the identity map and that flip an edge (i.e., there is an edge in the graph that gets mapped onto itself with its orientation reversed). 

The purpose of these comments is to show that the arguments that we have considered so far generalize in a natural way to maps on graphs. These arguments about periodic orbits can often be reduced to elementary arguments involving the traces of matrices, and these are accessible to undergraduates. Indeed, \cite{BGJR} was written as part of an REU project at Fairfield University.

The arguments that we use are elementary, but they are essentially homological in nature. The Oriented Markov matrix is the matrix for the map on one-chains. The Lefschetz number can be calculated using the alternating sum of the traces of the matrices of the chain groups. Since the Lefschetz number is a homotopy invariant, we could use the Lefschetz number to calculate the trace of the Oriented Markov matrix (see \cite{B5} for more explanation on this), but this approach is less elementary.

\section{Completion of proof of Sharkovsky's theorem}

As noted above, to complete the proof we need to show that if $f$ has a periodic point of least period $2^k$ times a odd integer then it has a periodic point with least period $2^{k+1}3$. In the first part of the proof we used the fact that a negatively oriented, non-repetitive closed walk of length $m$ in the Oriented Markov Graph of $L_\theta$ tells us that there must be periodic points with least period $m$ for both $f$ and $L_\theta$. Our next result shows that in certain cases we can use periodic points of $L_\theta$ to deduce the existence of closed walks in the Oriented Markov Graph.

Suppose we are given a connect-the-dots map $L_\theta$ and a periodic point $p$ with least period $m$ such that none of the iterates of $p$ are critical points of $L_\theta$. Then $L_\theta^m$ is differentiable at $p$ and so we can assign an orientation to $L_\theta^m$ at $p$ given by the sign of the derivative. 

\begin{lemma}
Let $L_\theta$ be a connect-the-dots map. Suppose there exists a periodic point $p$ with least period $m$ such that none of the iterates of $p$ are critical points of $L_\theta$  and such that $L_\theta^m$ at $p$ has negative orientation. Then there is a non-repetitive closed walk in the Oriented Markov Graph of length $m$ with negative orientation.
\end{lemma}

\begin{proof}
Suppose that $L_\theta^i(p) \in E_{j_i}$ and $L_\theta^{i+1}(p) \in E_{j_{i+1}}$. Then a closed subinterval $J_{i} \subseteq E_{j_i}$ can be found such that $p \in J_i$ and such that $f(J_i)=E_{j_{i+1}}$. This means that there is a directed edge from $E_{j_i}$ to $E_{j_{i+1}}$ in the Markov Graph of $L_\theta$. Since $f$ restricted to $J_i$ has constant slope that is non-zero we can assign an orientation to get the appropriate oriented directed edge from $E_{j_i}$ to $E_{j_{i+1}}$ in the Oriented Markov Graph of $L_\theta$. This observation shows that there must be a closed walk in the Oriented Markov Graph that has length $m$ and has negative orientation. Now we must show that this walk is non-repetitive.

Suppose for a contradiction that the closed walk of length $m$ is the repetition of a shorter walk of length $k$, say. Then there exists a closed subinterval $K_0 \subseteq E_{i_0}$ such that $L_\theta^k(K_0)=E_{i_0}$. Notice that if $x$ has the property that $L_\theta^i(x) \in E_{j_i}$ for $0 \leq i \leq k$ then $x \in K_0$. This implies that $L_\theta^{ik}(p)\in K_0$ for all non-negative integers $i$.  Also note that $L_\theta^k$ has constant slope on $K_0$ and that this must negative. But $p$ is a periodic point of $L_\theta^k$ on $K_0$. There are only two possible ways of obtaining a periodic orbit on an interval with a map of constant negative slope, both of which result in a contradiction. Either the point is fixed, which would mean that $L_\theta^k(p)=p$ contradicting the fact that $k<p$, or the slope is $-1$ and $p$ has period two under $L_\theta^k$, but this means that $m=2k$ and that $p$ has positive orientation contradicting the fact that the orientation is negative.

 \end{proof}

To complete the proof of Sharkovsky's theorem we need the following technical result.

\begin{lemma} \label{horseshoe}
Let $f: \mathbb{R} \to \mathbb{R}$ be continuous. Suppose that there exists $a,b,c\in \mathbb{R}$ satisfying $f(b)\leq c<b<a=f(a)\leq f(c)$, then $f$ has periodic points of all least periods. These points can all be chosen so that the periodic orbits have negative orientation.
\end{lemma}

\begin{proof}
Label the closed intervals $[c,b]$ and $[b,a]]$ by  $E_1$ and $E_2$ respectively. A picture of the situation is shown in Figure 8 and the corresponding Oriented Markov Graph is shown in Figure 9.

\begin{centering}
\begin{figure}
\begin{tikzpicture}[scale=6]

\draw[thick] (-.5,0) to (1.5,0);
  \draw[]{(1.55, -.07)} node {$f(c)$};
  \draw[]{(-.03, -.07)} node {$c$};
  \draw[]{(0.5 ,-.07)} node {$b$};
    \draw[]{(1 ,-.07)} node {$a$};
      \draw[]{(-.5, -.07)} node {$f(b)$};
    
     \draw[]{(.25 ,-.05)} node {$E_1$};
      \draw[]{(.75 ,-.05)} node {$E_2$};
  
   \draw[fill=red] (0, 0) circle (.3pt);
    \draw[fill=red] (.5, 0) circle (.3pt);
     \draw[fill=red] (1.5, 0) circle (.3pt);
        \draw[fill=red] (1, 0) circle (.3pt);
         \draw[fill=red] (-.5, 0) circle (.3pt);
  
  \draw[dotted, thick, ->, blue] (.49,-.01) to [bend left=45] (-0.48,-.01);
  \draw[dotted, thick, ->, blue] (1, -.01) ..controls (1.5,-.5) and (.5,-.5).. (.99,-.02);
    \draw[dotted, thick, ->, blue] (0, .01) to [bend left=30] (1.49,0.01);

 \end{tikzpicture}
 
 \caption{Diagram for Lemma $5$}
 \end{figure}
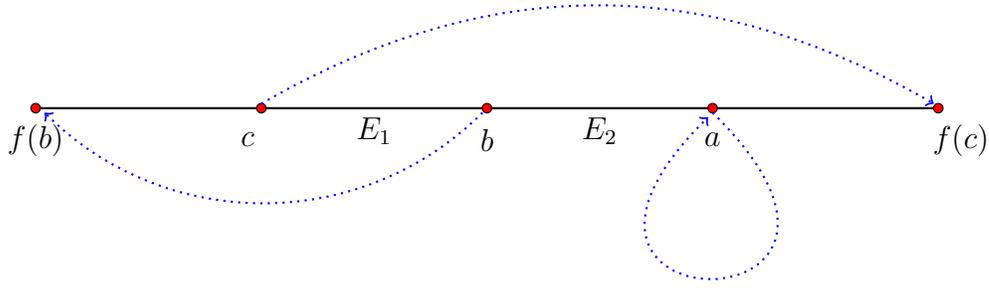
  \end{centering}

  \begin{centering}
\begin{figure}
\begin{tikzpicture}[scale=5]

  \draw[]{(-.02, 0)} node {$E_1$};
  \draw[]{(1.02,0)} node {$E_2$};

  \draw[ thick, ->] (0.01,0.04) to [bend left=30] node[font=\tiny, fill=white] {$-$}(.99,.04);
  \draw[ thick, ->] (1,-.04) to [bend left=30] node[font=\tiny, fill=white] {$+$}(.01,-.04);
    \draw[ thick, ->] (-.04,0.04)..controls (-.3,+.4) and (-.3,-.4)..  node[font=\tiny, fill=white] {$-$} (-.04,-.04);
     \draw[ thick, ->] (1.04,0.04)..controls (1.3,.4) and (1.3,-.4)..  node[font=\tiny, fill=white] {$+$} (1.04,-.04);

 \end{tikzpicture}

\caption{Oriented Markov Graph}
 
\end{figure}
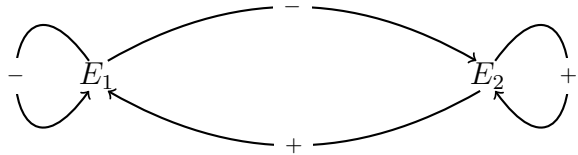
\end{centering}

It is straightforward to check that, given any positive integer $k$, we can find a non-repetitive, closed walk of length $k$ with negative orientation.
\end{proof}

\begin{theorem} \label{2diag}
Let $\theta$ be a cycle of length $n$. If the oriented Markov matrix has more than one non-zero entry on the main diagonal, then $L_\theta$ has periodic points of all least periods. These periodic points can all be chosen so that they have negative orientation.

\end{theorem}

\begin{proof}
Suppose that the oriented Markov matrix has more than one non-zero entry on the main diagonal, then the graph of $L_\theta$ will cross the line $y=x$ more than once. Let the leftmost intersection of the graph of $L_\theta$ with $y=x$ be denoted by $z$ and let the next intersection be denoted by $a$. Clearly, $L_\theta$ has negative orientation at $z$ and positive at $a$. Let $b$ denote the $x$-coordinate where $L_\theta$ achieves its absolute minimum on the interval $[z,a]$. Since $L_\theta$ is a connect-the-dot map, $b$ must be an integer in $\{1, 2, \dots, n\}$. This means that some iterate of $b$ must be greater than $a$. It is also clear that $L_\theta(b)<b$. 

 Consider the interval $[b,a]$. This gets mapped onto the interval $[L_\theta(b),a]$ under $L_\theta$. Now consider the image of $[L_\theta(b),a]$ under $L_\theta$. Points in the half-open interval $[L_\theta(b),z)$ get mapped to the right,  and points in the open interval $(z, a)$ get mapped to the left. So the leftmost point of $[L_\theta(b),a]$ under $L_\theta$ is $L_\theta(b)$. If there is no point in $[L_\theta(b),z)$ that gets mapped to the right of $a$ then $[L_\theta(b),a]$ must get mapped onto itself under $L_\theta$, but this would mean that no iterate of $b$ could be greater than $a$. So there must be some point $c$ satisfying $L_\theta(b)\leq c <z$ with $L_\theta(c)>a$. Lemma \ref{horseshoe} completes the proof.

\end{proof}

We are now in a position to complete the proof of Sharkovsky's theorem. The last result that we need is:

\begin{lemma}\label{realshark}
If $n=2^kr$ where $k \geq 0$ and $r>1$ is a positive odd integer then $L_\theta$ has a periodic point of least period $2^{k+1}3$.

\end{lemma}

\begin{proof}
The permutations $\theta^{2^k}$ and $\theta^{2^{k+1}}$ do not fix any of the integers in \\ $\{1, 2, \dots, n\}$, so the traces of the Oriented Markov Matrices of $L_\theta^{2^k}$ and $L_\theta^{2^{k+1}} $ are both $-1$. 

If the main diagonal of the Oriented Markov Matrix of $\theta^{2^{k+1}}$ has more than one non-zero entry on its main diagonal then Theorem \ref{2diag} tells us that $L_{\theta^{2^{k+1}}}$ has a periodic point of period three with negative orientation. Let $a$ denote an element of this orbit. Then we know that $a$ must be fixed by $L_{\theta^{2^{k+1}}}^3$ and so that the least period of $a$ with respect to $L_\theta$ must divide $2^{k+1}3$ and that it must also be divisible by $3$. This means that the least period of $a$ is $2^r3$ for some $r\leq k+1$. Lemma \ref{2kr} completes the proof in this case.

We now consider the case when the  Oriented Markov Matrix of $\theta^{2^{k+1}}$ only one non-zero entry on its main diagonal. Since the trace is $-1$ we know that this entry must be $-1$. We also know that the trace of the Oriented Markov Matrix of $\theta^{2^{k}}$ is $-1$. So there is a closed walk of length $2^k$ in the oriented Markov graph of $L_\theta$ with negative orientation. Let $E_j$ be a vertex in the Markov graph that is part of this closed walk. Notice that repeated twice this walk becomes a (repetitive) closed walk of length $2^{k+1}$ with positive orientation. This means that it contributes $+1$ to the main diagonal of the Oriented Markov Matrix of $\theta^{2^{k+1}}$. Since there is only one non-zero entry on the main diagonal and it is $-1$ there must be another closed walk of length $2^{k+1}$ from $E_j$ with negative orientation. These two closed walks from $E_j$ to itself must involve at least two prime closed walks.

As in the proof of Lemma $3$, we can use these prime walks to construct a non-repetitive closed walk of length $2^{k+1}3$ with negative orientation.
So $L_\theta$ has a periodic point of least period $2^{k+1}3$.

\end{proof}

\section{Second digression}

Earlier we showed that the trace arguments extend to maps on trees and to maps on graphs in certain cases. We now give an example to show that the results in the previous section do not extend to these cases. Consider the map on the tree shown in Figure 10 that maps the vertices according to the permutation $(1,2,3,4,5,6,7,8,9)$. The  Markov Graph is shown in Figure 11.

\begin{centering}
\begin{figure}
\begin{tikzpicture}[scale=5]

\draw[thick] (0,0) to node[below, font=\tiny]{$E_2$ }(.5 ,0);
\draw[thick] (.5,0) to node[below, font=\tiny]{$E_1$ }(1 ,0);
\draw[thick] (1,0) to node[below, font=\tiny]{$E_3$ }(1.5 ,0);
\draw[thick] (1,0) to node[left, font=\tiny]{$E_5$ }(1 ,.3);
\draw[thick] (0,0) to node[left, font=\tiny]{$E_6$ }(0 ,.3);
\draw[thick] (.5,0) to node[left, font=\tiny]{$E_4$ }(.5 ,.3);
\draw[thick] (1.5,0) to node[left, font=\tiny]{$E_7$ }(1.5 ,.3);
\draw[thick] (.5,0) to node[left, font=\tiny]{$E_8$ }(.5 ,-.3);

  \draw[fill=red] (0, 0) circle (.3pt);
    \draw[fill=red] (.5, 0) circle (.3pt);
      \draw[fill=red] (1, 0) circle (.3pt);
        \draw[fill=red] (1.5, 0) circle (.3pt);
        \draw[fill=red] (1, .3) circle (.3pt);
           \draw[fill=red] (0, .3) circle (.3pt);
              \draw[fill=red] (.5, .3) circle (.3pt);
                 \draw[fill=red] (1.5, .3) circle (.3pt);
                    \draw[fill=red] (.5, -.3) circle (.3pt);
        
        \draw {(0,-.05)} node[blue, font=\tiny]{$3$};
        
        \draw {(.45,-.05)} node[blue, font=\tiny]{$1$};
         \draw {(1,-.05)} node[blue, font=\tiny]{$2$};
          \draw {(1.5,-.05)} node[blue, font=\tiny]{$4$};
    \draw {(-.05,.3)} node[blue, font=\tiny]{$7$};
      \draw {(.45,.3)} node[blue, font=\tiny]{$5$};
        \draw {(.95,.3)} node[blue, font=\tiny]{$6$};
          \draw {(1.45,.3)} node[blue, font=\tiny]{$8$};
             \draw {(.45,-.3)} node[blue, font=\tiny]{$9$};

 \end{tikzpicture}
 \caption{Tree map}
\end{figure}
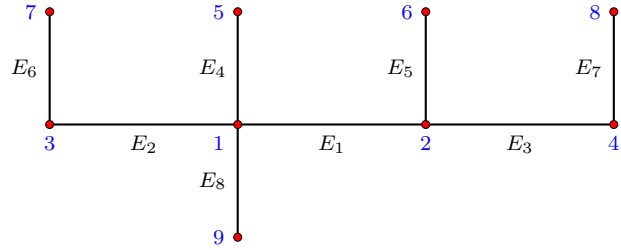
\end{centering}

\begin{centering}
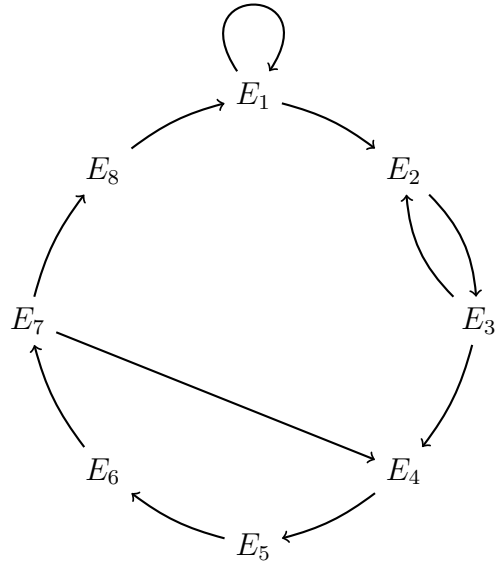
\begin{figure}
\begin{tikzpicture}[scale=5]

\draw{(0,0)} node (5) {$E_5$};
\draw{(-.4,.2)} node (6) {$E_6$};
\draw{(.4,.2)} node (4) {$E_4$};
\draw{(-.6,.6)} node (7) {$E_7$};
\draw{(.6,.6)} node (3) {$E_3$};
\draw{(-.4,1)} node (8){$E_8$};
\draw{(.4,1)} node (2){$E_2$};
\draw{(0,1.2)} node (1){$E_1$};

\draw[thick, ->] (1) to [bend left=11] (2);
\draw[thick, ->] (2) to [bend left=20] (3);
\draw[thick, ->] (3) to [bend left=11] (4);
\draw[thick, ->] (4) to [bend left=11] (5);
\draw[thick, ->] (5) to [bend left=11] (6);
\draw[thick, ->] (6) to [bend left=11] (7);
\draw[thick, ->] (7) to [bend left=11] (8);
\draw[thick, ->] (8) to [bend left=11] (1);

\draw[thick, ->] (3) to [bend left=20] (2);
\draw[thick, ->] (7) to (4);
\draw[thick, ->] (1) .. controls (-.2,1.5) and (.2, 1.5)..(1);

 \end{tikzpicture}
 \caption{Markov Graph for tree map}
\end{figure}
\end{centering}
 \noindent Looking at the Markov Graph it is clear that there are no closed walks of length $6$. This means that the connect-the-dot maps has no periodic orbits with least period $6$. So Lemma \ref{realshark} does not hold if we change the underlying graph from being topologically an interval to being a general tree and so we cannot expect the Sharkovsky ordering to extend to more general cases.

 However, we can prove more about maps on trees, but before we can state the main theorem in this case, we have to describe the process of removing ones from the right.
 
\subsection{Removing ones from the right}\index{removing ones from right}

The process of removing ones from the right can be described as follows.

\begin{enumerate}
\item Write $n$ in binary.
\item Change the rightmost one in its expansion to zero. 
\item Repeat the process until you end with $0$.
\end{enumerate}

 For example, $31$ has binary expansion $11111$. Applying the process to this number yields the following binary expansions $11110$, $11100$, $11000$, $10000$ and $00000$, or in decimal notation $30$, $28$, $24$,  $16$ and $0$.

We can now state the theorem for general trees.

\begin{theorem} \label{t2}
Let $T$ be a tree with $v$ vertices. Let $f: T \rightarrow T$ be a map with the property that the vertices form one periodic orbit. Then:

\begin{enumerate}
\item If $v$ is not a divisor of $2^k$, then $f$ has a periodic point with least period $2^k$.  
\item If $v=2^pq$, where $q>1$ is odd and $p \geq 0$, then $f$ has a periodic point with least period $2^pr$ for any $r \geq q$.
\item The map $f$  also has periodic orbits of any least period $m$ where $m$ can be obtained from $v$ by removing ones from the right of the binary expansion of $n$ and changing them to zeros.
\end{enumerate}
\end{theorem}

This ordering can be visualized, and is shown in Figure 12.

\begin{centering}
\begin{figure}
\begin{tikzpicture}[scale=4]

\draw{(2,.6)} node (3) {$3$};
\draw{(2,.4)} node (5) {$5$};
\draw{(2 ,.2)} node (7) {$7$};
\draw{(2 ,0)} node (9) {$9$};
\draw{(2 ,-.2)} node (11) {$11$};
\draw{(2 ,-.4)} node (13) {$13$};
\draw{(2 ,-.6)} node (15) {$15$};
\draw{(2 ,-.8)} node (17) {$17$};
\draw{(2 ,-1)} node (19) {$19$};
\draw{(2 ,-1.2)} node (21) {$21$};
\draw{(2 ,-1.4)} node (23) {$23$};
\draw{(2 ,-1.6)} node (25) {$25$};
\draw{(2 ,-1.8)} node (27) {$27$};
\draw{(2 ,-2)} node (29) {$29$};
\draw{(2 ,-2.2)} node (31) {$31$};

\draw{(1.5,.2)} node (6) {$6$};
\draw{(1.5,-.2)} node (10) {$10$};
\draw{(1.5,-.6)} node (14) {$14$};
\draw{(1.5 ,-1)} node (18) {$18$};
\draw{(1.5 ,-1.4)} node (22) {$22$};
\draw{(1.5 ,-1.8)} node (26) {$26$};
\draw{(1.5 ,-2.2)} node (30) {$30$};

\draw{(1,-.6)} node (12) {$12$};
\draw{(1,-1.4)} node (20) {$20$};
\draw{(1,-2.2)} node (28) {$28$};

\draw{(.5,-2.2)} node (24) {$24$};

\draw [thick, dotted] (31)  to (2,-2.7);
\draw [thick, dotted] (30)  to (1.5,-2.7);
\draw [thick, dotted] (28)  to (1,-2.7);
\draw [thick, dotted] (24)  to (.5,-2.7);

\draw [thick, dotted] (0,-2.5)  to (0,-2.7);
\draw [thick, dotted] (-.5,-2.5)  to (-.5,-2.7);

\draw[thick, ->] (7) to (6);
\draw[thick, ->] (11) to (10);
\draw[thick, ->] (15) to (14);
\draw[thick, ->] (19) to (18);
\draw[thick, ->] (23) to (22);
\draw[thick, ->] (27) to (26);
\draw[thick, ->] (31) to (30);
\draw[thick, ->] (30) to (28);

\draw[thick, ->] (14) to (12);
\draw[thick, ->] (22) to (20);
\draw[thick, ->] (23) to (22);

\draw[thick, ->] (28) to (24);

\draw[thick, ->] (3) to (5);
\draw[thick, ->] (5) to (7);
\draw[thick, ->] (7) to (9);
\draw[thick, ->] (9) to (11);
\draw[thick, ->] (11) to (13);
\draw[thick, ->] (13) to (15);
\draw[thick, ->] (15) to (17);
\draw[thick, ->] (17) to (19);
\draw[thick, ->] (19) to (21);
\draw[thick, ->] (21) to (23);
\draw[thick, ->] (23) to (25);
\draw[thick, ->] (25) to (27);
\draw[thick, ->] (27) to (29);
\draw[thick, ->] (29) to (31);

\draw[thick, ->] (6) to (10);
\draw[thick, ->] (10) to (14);
\draw[thick, ->] (14) to (18);
\draw[thick, ->] (18) to (22);
\draw[thick, ->] (22) to (26);
\draw[thick, ->] (26) to (30);

\draw[thick, ->] (12) to (20);
\draw[thick, ->] (20) to (28);

\draw{(-.8,-3)} node (1) {$1$};
\draw{(-.5,-3)} node (2) {$2$};
\draw{(-.2,-3)} node (4) {$4$};
\draw{(.1,-3)} node (8) {$8$};

\draw[thick, ->] (8) to (4);
\draw[thick, ->] (4) to (2);
\draw[thick, ->] (2) to (1);

\draw [thick, dotted] (8)  to (1,-3);

\draw [thick] (-.8, -2.8)  to (-.8,-2.9);
\draw [thick] (-.8, -2.9)  to (2,-2.9);
\draw [thick, ->] (2,-2.9)  to (2,-3);

\draw [thick, dotted] (-.8, -2.6)  to (-.6,-2.6);

 \end{tikzpicture}
 \caption{Ordering for trees}
\end{figure}
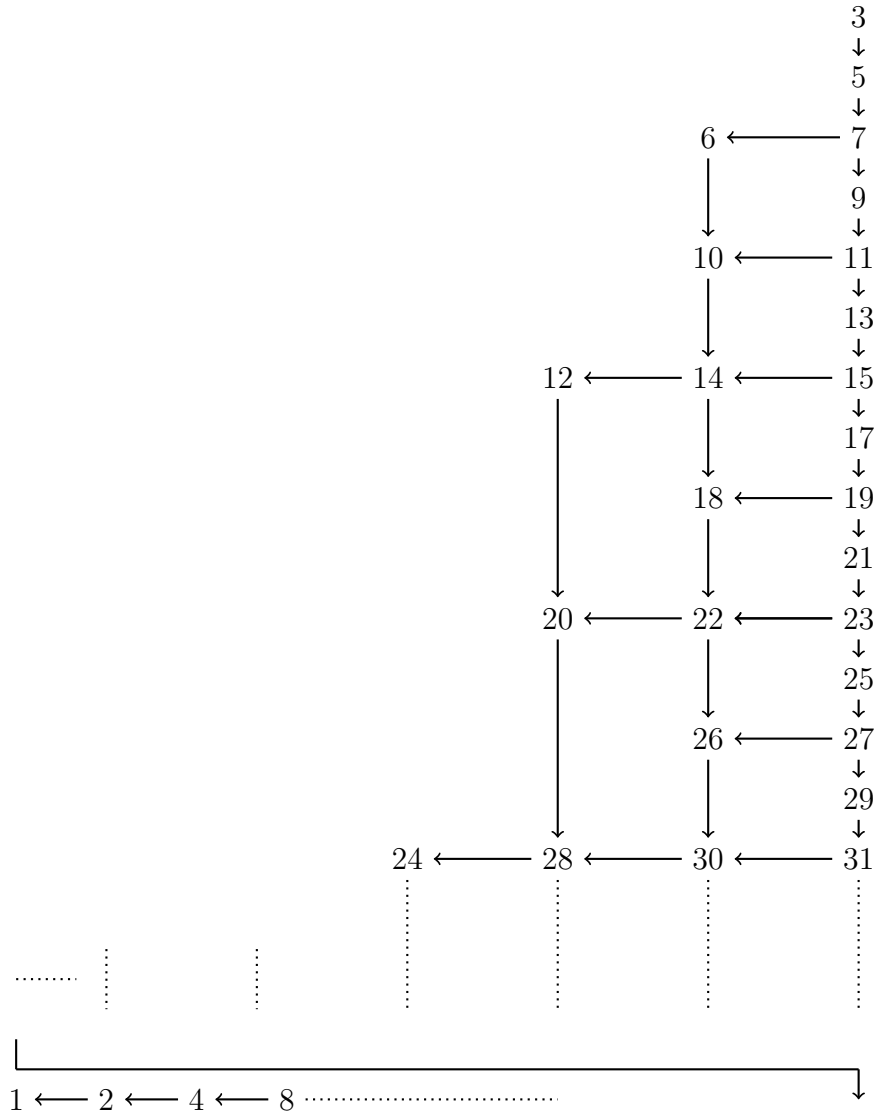
\end{centering}

\subsection{Converses}

As stated in the introduction, Sharkovsky also proved what is often called the converse, that given any positive integer, there exists a map from the reals to itself that has periodic points with only the least periods that are given by the ordering. In \cite{B2, B3} the converse result was also proved for Theorem \ref{t2}. However, as yet, there are no converse results to the theorems on graphs. It is an open question as to whether maps on graphs that are homotopic to the constant map have the tree ordering. Similarly it is an open question as to whether maps on graphs that are homotopic to the identity and flip an edge have the tree ordering.

\section{Concluding remarks}

The purpose of this paper was to give a proof of Sharkovsky's theorem, but also to show that there are many ways that these ideas can be generalized to maps on graphs. This leads to many questions that are amenable to undergraduate research. For example, what can be said when the vertices of a graph form more than one periodic orbit (see \cite{B4} for a beginning result in this area)? How do you construct a graph and map to have a prescribed set of periods? 

It should also be pointed out that the results involving the trace of matrices are closely related to the Lefschetz number of the map. Indeed, the examples in this paper seem to be nice introductory examples that could be used in a beginning algebraic topology course to show why the Lefschetz number is important and how it can be used. 

Finally, maps on trees and graphs can be studied in much more generality, though the arguments are much more intricate and generally not as accessible for undergraduates. See \cite{AJM} for a major result on trees and \cite{LL, LPR} for some results on graphs. For an introduction to combinatorial dynamics the book \cite{ALM} is the basic reference.

%%%%%%%%%%%%

\end{document}